\documentclass[11pt]{amsart}
\usepackage{amscd,amssymb,verbatim}
\theoremstyle{plain}
\newtheorem{Thm}{Theorem}[section]
\newtheorem{Cor}[Thm]{Corollary}
\newtheorem{Lem}[Thm]{Lemma}

\newtheorem{Def}[Thm]{Definition}

\theoremstyle{definition}
\newtheorem{remark}{Remark}

\numberwithin{equation}{section}

\begin{document}
\title[An extension of James's compactness Theorem]{An extension of James's compactness Theorem}
\author{I. Gasparis}
\address{School of Applied Mathematical and Physical Sciences \\
National Technical University of Athens\\
Athens, 15708 \\
Greece}
\email{ioagaspa@math.ntua.gr}
\begin{abstract}
Let \(X\) and \(Y\) be Banach spaces and \(F \subset B_{Y^*}\). Endow \(Y\) with
the topology \(\tau_F\) of point-wise convergence on \(F\). Assume that \(T \colon X^* \to Y\)
is a bounded linear operator 
which is \((w^*, \tau_F)\) continuous. 
Assume further that every vector in the range of \(T\) attains its norm at some element of \(F\) (that is, for every \(x^* \in X^*\) there
exists \(y^* \in F\) such that \(\|T(x^*)\| = |y^*(Tx^*)|\)).
Then \(T\) is \((w^*,w)\) continuous. The proof relies on 
Rosenthal's \(\ell_1\)-theorem. As a corollary to the above result, one obtains an alternative proof of James's compactness theorem
that a bounded subset \(K\) of a Banach space \(E\) is relatively weakly compact provided that each functional in \(E^*\) attains
its supremum on \(K\). 
\end{abstract}
\maketitle
\section{Introduction}
James's famous characterization of reflexivity \cite{J} states that a Banach space \(X\) is reflexive if, and only if, every bounded linear functional
on \(X\) attains its norm at an element of \(B_X\). Subsequently, James \cite{J1} characterized the weakly compact subsets of a Banach space as the weakly closed,
bounded subsets on which every bounded linear functional attains its supremum. James's proofs were technically quite demanding and there has been a considerable effort made
for discovering a simpler proof (\cite{P}, \cite{D}, \cite{RO}, \cite{G1}, \cite{G2}, \cite{FL}, \cite{MO}, \cite{K}, \cite{Pf1}, \cite{Pf2}, \cite{M}, \cite{C}). 

When the underlying Banach space is separable, elegant proofs have been given in
\cite{G1} (cf. also \cite{G2}), using Simons's inequality \cite{S}, and in \cite{FL} using convexity methods. The latter ones were refined in \cite{M} to cover
spaces having \(w^*\)-sequentially compact dual balls. The situation changes drastically in the non-separable case where the arguments become more delicate
mainly due to the fact that the \(w^*\) topology on bounded subsets of the dual is no longer metrizable. The methods of \cite{FL} have been extended in \cite{K}
to cover the non-separable case as well. In \cite{Pf1} (cf. also \cite{Pf2}) Godefroy's boundary problem is answered in the affirmative yielding a new proof of James's theorem based on
Simon's inequality, Rosenthal's \(\ell_1\)-theorem \cite{R} and a refinement of a technique due to J. Hagler and W. B. Johnson \cite{HJ}
for extracting \(\ell_1\)-sequences in spaces whose duals contain \(\ell_1\)-sequences without \(w^*\)-convergent subsequences. The proof of James's theorem given  
in \cite{MO} was the first one to combine the results from \cite{S}, \cite{R} and \cite{HJ}. We finally mention paper \cite{C} where the arguments
of \cite{P} are extended to give quantitative versions of James's theorem.

We shall next describe the main result of this paper. The starting point is the following result of D. Amir and J. Lindenstrauss \cite{AL}: Suppose that
\(X\) is a Banach space generated by a weakly compact subset \(K\). Let \(C(K)\) denote the Banach space of scalar-valued functions, continuous on \(K\), under
the supremum norm.
Consider the natural restriction operator \(R \colon X^* \to C(K)\) given by
\(R(x^*) = x^* | K\). Then \(R\) is \((w^*, w)\) continuous. In particular, \(R\) is weakly compact. 

We first observe
that some sort of converse to the previous result holds as well. 
Indeed, let \(K\) be a bounded subset of the Banach space \(X\). After naturally
identifying \(X\) with a closed subspace of \(X^{**}\), we set \(L = \overline{K}^{w^*}\)
and let \(R \colon X^* \to C(L)\) be the natural restriction operator. It is shown in Corollary \ref{C2} that if \(R\) is weakly compact
then \(L \subset X\) and so \(K\) is relatively weakly compact. Now suppose \(K\) satisfies the hypotheses in the statement of James's theorem.
The result will follow once we manage to show that \(R\) is weakly compact. 
Let \(Y = C(L)\) and identify \(K\) with a subset of \(B_{Y^*}\). 
Then the restriction operator \(R \colon X^* \to Y\) satisfies the following two properties: the first one is that
\(R\) is \((w^*, \tau_K)\) continuous, where \(\tau_K\) denotes the topology in \(Y\) of point-wise convergence on \(K\). 
The second property that \(R\) satisfies is that every vector in the range of \(R\) attains its norm at an element of \(K\).
Our main result states that these properties suffice to ensure that \(R\) is weakly compact. More precisely we prove the following theorem.
\begin{Thm} \label{MTh}
Let \(X\) and \(Y\) be Banach spaces and \(F \subset B_{Y^*}\). Endow \(Y\) with
the topology \(\tau_F\) of point-wise convergence on \(F\). Assume that \(T \colon X^* \to Y\)
is a bounded linear operator 
which is \((w^*, \tau_F)\) continuous. 
Assume further that for every \(x^* \in X^*\) there
exists \(y^* \in F\) such that \(\|T(x^*)\| = |y^*(Tx^*)|\).
Then \(T\) is \((w^*,w)\) continuous. 
\end{Thm}
The proof relies on Rosenthal's \(\ell_1\)-theorem and James's distortion theorem \cite{J2}. It does not make use of Simons's inequality.
A key role in the proof is played by the \(\ell_1^+\)-sequences. A normalized sequence \((x_n)\) in a Banach space is an \(\ell_1^+\)-sequence
if there is some \(c > 0\) such that \(\|\sum_n a_n x_n \| \geq c \sum_n a_n\) whenever \((a_n) \in \ell_1\) with \(a_n \geq 0\) for all \(n \in \mathbb{N}\).
It is easy to see that every normalized, non-weakly null sequence in a Banach space, admits an \(\ell_1^+\)-subsequence. 

It is shown in Lemma \ref{L1} that if \((x_n)\) is an \(\ell_1^+\)-sequence then there exists a sequence of positive scalars \((b_n)\) with \((b_n) \in \ell_1\)
and such that \((1 + b_n)\|\sum_{i=1}^n b_i x_i \| = \|\sum_{i=1}^{n+1} b_i x_i \| \) for all \(n \in \mathbb{N}\). This result serves as a
substitute to Simons's inequality and enables us to prove Corollary \ref{C1}, first proved by Simons through his inequality, a Rainwater-type of result
where the set of extreme points of the dual ball is replaced by any boundary of the space.
Corollary \ref{C1} yields a rather elementary proof of James's theorem for spaces having \(w^*\)-sequentially compact dual balls. These results are presented
in Section \ref{S1}. 

The proof of Theorem \ref{MTh} is given in Section \ref{S2}. We first prove Lemma \ref{L2}, a refinement of
the key lemma \ref{L1} for sequences generating \(\ell_1^+\) almost isometrically. This lemma is helpful in the study of the \(F\)-admissible subsets 
of a Banach space \(Y\), where \(F \) is a symmetric subset of \(B_{Y^*}\). A subset \(K\) of \(Y\) is \(F\)-admissible if it is bounded, compact in the topology of point-wise
convergence on \(F\) and if \(\sum_n a_n y_n \) attains its norm at an element of \(F\) for every \((a_n) \in \ell_1\) and every sequence \((y_n)\) in \(K\).
Note that, under the assumptions of Theorem \ref{MTh}, \(T(A)\) is \(F\)-admissible for every \(w^*\)-compact subset \(A\) of \(X^*\). The main point about
these sets is that they can not contain normalized sequences spanning \(\ell_1\) almost isometrically. This is shown in Corollary \ref{C4} by combining
Rosenthal's \(\ell_1\)-theorem, James's distortion theorem and Corollary \ref{C3}.
Theorem \ref{MTh} will then follow because \(T\) maps \(w^*\)-compact subsets of \(X^*\) into weakly compact subsets of \(Y\).
\section{Preliminaries}
We use standard Banach space facts and terminology as may be found in \cite{LT}.
Set \(\mathbb{T} = \{z \in \mathbb{C}: \, |z| = 1\}\). A subset \(A\) of a Banach space is {\em symmetric}
if \(z A \subset A\) for all \(z \in \mathbb{T}\).

Let \(X\) be a Banach space. 
\(B_{X}\) stands for the closed unit ball of \(X\). A boundary for \(X\) is a subset \(B\) 
of \(B_{X^*}\) with the property that for every \(x \in X\) there exists \(b^* \in B\) such that \(|b^*(x)| = \|b\|\).

A sequence in \(X\) is called an \(\ell_1\)-{\em sequence} if it is a basic sequence equivalent to the usual basis of \(\ell_1\). 
A normalized sequence \((e_n)\) in \(X\) is said to generate \(\ell_1\)
{\em almost isometrically} if for every \(\epsilon > 0\) there exists \(n_0 \in \mathbb{N}\)
so that \(\|\sum_{n \geq n_0} a_n e_n \| \geq (1 + \epsilon)^{-1} \sum_n |a_n|\)
for all \((a_n)_{n=n_0}^\infty \in \ell_1\). If \(\Gamma \subset X\) then \(\Gamma\) is said to
generate \(\ell_1(|\Gamma|)\) {\em isometrically} if for \(n \in \mathbb{N}\), all choices of pairwise distinct members
\(x_1, \dots, x_n\) of \(\Gamma\) and all choices of scalars \(a_1, \dots, a_n\) we have that
\(\|\sum_{i=1}^n a_i x_i \| = \sum_{i=1}^n |a_i|\). 

If \(F \subset B_{X^*}\) then the topology \(\tau_F\) of {\em point-wise} convergence
on \(F\) is a linear topology on \(X\) having as a neighborhood basis of the origin the collection of sets
\(\{U(G, \epsilon): \, G \subset F \, \text{ finite} , \, \epsilon > 0\}\) where,
\(U(G,\epsilon) = \{x \in X: \, |b^*(x)| < \epsilon, \, \forall \, b^* \in G \}\).

Let \((Z, \tau)\) be a topological space and \((z_n)\) be a sequence in \(Z\). A \(\tau\)-{\em cluster point}
of \((z_n)\) is any limit of a \(\tau\)-convergent subnet of \((z_n)\). 

If \(M\) is an infinite subset of \(\mathbb{N}\) then the notation \(M=(m_n)\) indicates the increasing enumeration of \(M\).
\([M]\) stands for the set of all infinite subsets of \(M\). 

We let \(\mathcal{T}\) denote the dyadic tree \(\cup_{n=1}^\infty \{0,1\}^n\). \(\mathcal{T}\) is partially ordered by
initial segment inclusion that is, \((a_1 , \dots, a_n ) \leq (b_1 , \dots , b_m)\) in \(\mathcal{T}\)
precisely when \(n \leq m\) and \(a_i = b_i\) for all \(i \leq n\). A branch of \(\mathcal{T}\) is a maximal, under inclusion,
well-ordered subset. A tree of infinite subsets of \(\mathbb{N}\) is a collection \((M_\alpha)_{\alpha \in \mathcal{T}}\)
of members of \([\mathbb{N}]\), indexed by \(\mathcal{T}\), so that \(M_\beta \subset M_\alpha\) whenever \(\alpha \leq \beta\)
while \(M_\alpha \cap M_\beta = \emptyset\) if \(\alpha\) and \(\beta\) are incomparable.
\section{A proof of James's theorem in the separable case} \label{S1}
We let \(\ell_1^+\) denote the positive cone of \(\ell_1\) that is, 
\(\ell_1^+ = \{(a_n) \in \ell_1 : \, a_n \geq 0, \, \forall \, n \in \mathbb{N}\}\).
\begin{Def}
\begin{enumerate}
\item A normalized sequence \((e_n)\) in a Banach space is called an \(\ell_1^+\) sequence
if there is some \(c > 0\) such that \[\|\sum_n a_n e_n \| \geq c \sum_n a_n\]
for all \((a_n) \in \ell_1^+\). We then say that \((e_n)\) is an c-\(\ell_1^+\) sequence.
In case \(c=1\), \((e_n)\) is said to generate \(\ell_1^+\) isometrically. 
\item A normalized sequence \((e_n)\) in a Banach space is said to generate \(\ell_1^+\) almost
isometrically, if for all \(\epsilon > 0\) there exists \(n_0 \in \mathbb{N}\) so that
\((e_n)_{n \geq n_0}\) is an \((1+ \epsilon)^{-1}\)-\(\ell_1^+\) sequence.
\end{enumerate}
\end{Def}
Let \((e_n)\) be a bounded, non-weakly null sequence of non-zero vectors in a Banach space. Then, evidently,
\((e_n /\|e_n\|)\) admits an \(\ell_1^+\) subsequence.
\begin{Lem} \label{L1}
Let \(Y\) be a Banach space and \((y_n)\) a normalized \(\ell_1^+\) sequence in \(Y\).
Then there exists a sequence of positive scalars \((b_n)\) satisfying the following property:
\[ (1+b_n) \biggl \| \sum_{i=1}^n b_i y_i \biggr \| = \biggl \| \sum_{i=1}^{n+1} b_i y_i \biggr \| < 1, \, 
\forall \, n \in \mathbb{N}\]
\end{Lem}
\begin{proof}
Assume that \((y_n)\) is a \(c\)-\(\ell_1^+\) sequence for some \( 0 < c \leq 1\). 
We use induction on \(n \in \mathbb{N}\) to construct the desired scalar sequence \((b_n)\).
Choose \( 0 < b_1 < e^{-1/c}\). Then,
\[(1 + b_1) \| b_1 y_1 \| = b_1(1+b_1) < (1/e)(1 + 1/e) < 1\]
The intermediate value theorem now, applied on \(\phi_2 \colon [0, \infty) \to \mathbb{R}\) with \(\phi_2(t) = \|b_1 y_1 + t y_2\|\),
provides us some \(b_2 > 0\) such that \((1 + b_1) \| b_1 y_1 \| = \|b_1 y_1 + b_2 y_2 \| < 1\).
We next assume that \(n \geq 2\) and that the positive scalars \(b_1, \dots , b_n\) have been chosen to satisfy
\[ (1+b_k) \biggl \| \sum_{i=1}^k b_i y_i \biggr \| = \biggl \| \sum_{i=1}^{k+1} b_i y_i \biggr \| < 1, \, 
1 \leq k \leq n -1 \]
It follows now from our inductive assumption that
\[ (1+b_n) \biggl \| \sum_{i=1}^n b_i y_i \biggr \| = b_1 \prod_{i=1}^n (1+b_i) \leq b_1 e^{\sum_{i=1}^n b_i} \leq
b_1 e^{\frac{1}{c} \|\sum_{i=1}^n b_i y_i \|} \leq b_1 e^{\frac{1}{c}} < 1\]
We finally apply the intermediate value theorem for the function 
\(\phi_{n+1} (t) = \| \sum_{i=1}^n b_i y_i + t y_{n+1} \|\), \(t \geq 0\), to obtain \(b_{n+1} > 0\) such that
\[ (1+b_n) \biggl \| \sum_{i=1}^n b_i y_i \biggr \| = \biggl \| \sum_{i=1}^{n+1} b_i y_i \biggr \| < 1 \]
This completes the inductive step and the proof of the lemma. 
\end{proof}
The next result was obtained in \cite{S}. Here we present an alternative proof based on our previous lemma.
\begin{Cor} \label{C1}
Let \((y_n)\) be a bounded sequence in the Banach space \(Y\). Let \(F \subset B_{Y^*}\) be such that
\(\sum_n a_n y_n\) attains its norm at some element of \(F\) for all \((a_n) \in \ell_1^+\). Assume that
\(\lim_n y^* (y_n) = 0\) for all \(y^* \in F\). Then \((y_n)\) is weakly null.
\end{Cor}
\begin{proof}
We first observe that the hypotheses of the corollary are also fulfilled by any sequence of the form
\((y_{k_n}/\|y_{k_n}\|)\) where \((y_{k_n})\) is any subsequence of \((y_n)\) which is bounded away from zero.
Assume that \((y_n)\) is not weakly null. Our observation allows us to assume, without loss of generality, that
\((y_n)\) is a normalized \(\ell_1^+\) sequence. Let \((b_n)\) be the scalar sequence resulting from Lemma \ref{L1}
applied on \((y_n)\). Set \(u_n = \sum_{i=1}^n b_i u_i \), \(n \in \mathbb{N}\). Then, \((1+b_n) \|u_n \| = \|u_{n+1}\| < 1\)
for all \(n \in \mathbb{N}\). It follows that \(u = \sum_n b_n u_n \) is a well-defined element of \(Y\) and that
\(\|u\| = \|u_n\| \prod_{i=n}^\infty (1 + b_i)\) for all \(n \in \mathbb{N}\). Moreover, \(\|u_n \| < \|u_{n+1} \| < \|u\|\)
for all \(n \in \mathbb{N}\). Our hypotheses yield that \(u\) attains its norm at some element of \(F\). However, let \(y^* \in F\)
be arbitrary and pick \(0 < \delta < b_1/2\). Choose \(m \in \mathbb{N}\) so that \(|y^*(y_n) | < \delta\) and
\(b_n < 1\) for all \(n > m\). It follows that if we let \(t_n = \|u\|/\|u_n\|\), \(n \in \mathbb{N}\), then \((t_n)\) is strictly decreasing to \(1\)
and \(2 \delta /\|u\| < 1/ t_n \). We are thus led to the following estimates
\begin{align}
&|y^*(u)| \leq \|u_m \| + \sum_{n=m+1}^\infty b_n |y^*(y_n)| \leq \|u_m\| + \delta \sum_{n=m+1}^\infty b_n \notag \\
&\leq \|u_m\| + 2 \delta \sum_{n=m+1}^\infty \frac{b_n}{1 + b_n} 
\leq \|u_m\| + 2 \delta \sum_{n=m+1}^\infty \ln{(1 + b_n)} \notag \\
&= \|u_m \| + 2 \delta \ln{ \biggl [ \prod_{n=m+1}^\infty (1 + b_n) \biggr ] } 
= \|u_m\| + 2 \delta \ln{ \biggl (\frac{\|u\|}{\|u_{m+1}\|} \biggr )} \notag \\
&=\frac{\|u\|}{t_m} + 2 \delta \ln{ (t_{m+1})} \leq \frac{\|u\|}{t_m} + 2 \delta (t_m - 1) \notag \\
&=\|u\| \biggl [ \frac{1}{t_m} + \frac{2 \delta}{\|u\|} (t_m -1) \biggr ] < \|u\| \biggl ( \frac{1}{t_m} + \frac{t_m -1}{t_m} \biggr ) = \|u\| \notag
\end{align}
contrary to our assumptions. Therefore, \((y_n)\) is indeed weakly null.
\end{proof}  
Evidently, the preceding corollary readily yields Simons's result \cite{S} that every bounded sequence \((x_n)\) in a Banach space \(X\) satisfying
\(\lim_n b^* (x_n) = 0\) for all \(b^* \in B\), where \(B\) is a boundary for \(X\), is weakly null (cf. also \cite{M}). The case were \(B\) is the
set of the extreme points of \(B_{X^*}\) is the well known Rainwater's theorem \cite{RW}.

The next lemma is due to R. C. James \cite{J1} (cf. also \cite{O}) but we include a proof to be thorough.
\begin{Lem} \label{ES}
Let \(X\) be a Banach space and \(K \subset X\). We naturally identify \(X\) with a closed subspace of \(X^{**}\)
and assume that
\(\overline{K}^{w^*} \setminus X \ne \emptyset\).
Then there exist a sequence \((x_n) \)
in \(K\), a bounded sequence \((x_n^*)\) in \(X^*\) and \(\delta > 0\) so that for all \(n \in \mathbb{N}\)
\[ |x_n^*(x_i) | < 1/n , \, \forall \, i < n, \text{ yet }, |x_i^*(x_n) | > \delta,  \, \forall \, i \leq n\]
\end{Lem}
\begin{proof}
Choose \(x^{**} \in \overline{K}^{w^*} \setminus X\). 
The Hahn-Banach theorem yields \(x^{***} \in X^{***}\)
and \(\delta > 0\) so that \(|x^{***} (x^{**})| > \delta\) and \(x^{***} | X = 0\). 
We use induction on \(n \in \mathbb{N}\) to construct the desired sequences \((x_n)\) and \((x_n^*)\).
We first apply Goldstine's
theorem to obtain \(x_1^* \in X^*\) with \(\|x_1^*\| \leq \|x^{***}\|\) so that
\(|x^{**} (x_1^*)| > \delta\). We then choose \(x_1 \in K\) so that \(|x_1^*(x_1)| > \delta\).

We next assume that \(n \geq 2\) and that \((x_i)_{i=1}^{n-1} \subset K\), 
\((x_i^*)_{i=1}^{n-1} \subset X^*\) have been chosen so that
\begin{enumerate}
\item \(\|x_i^* \| \leq \|x^{***}\|\), \(|x_i^*(x_j)| < 1/i\), \(\forall j < i\), \(\forall i \leq n-1\).
\item \(|x_j^*(x_i) | > \delta\), \(\forall \, j \leq i \leq n-1\).
\item \(|x^{**} (x_i^*) | > \delta\), \(\forall \, i \leq n-1\).
\end{enumerate}
Since \(x^{***}\) vanishes on \(X\) and \(|x^{***} (x^{**})| > \delta\), Goldstine's theorem gives us some
\(x_n^* \in X^*\) such that \(\|x_n^* \| \leq \|x^{***}\|\), \(|x^{**} (x_n) | > \delta\) and
\(|x_n^* (x_j) | < 1/n\) for all \(j < n\). Therefore, \(|x^{**} (x_i^*) | > \delta \), for all
\(i \leq n\). On the other hand, \(x^{**} \in \overline{K}^{w^*}\) and so we may choose \(x_n \in K\)
such that \(|x_i^*(x_n) | > \delta\), for all \(i \leq n\). It follows now that 
\((x_i)_{i=1}^n \) and \((x_i^*)_{i=1}^{n}\) satisfy \((1)-(3)\) for \(n\). 
The inductive step and the proof of the lemma are now complete.
\end{proof}
\begin{remark}
The preceding lemma yields the less direct implication of the Eberlein-Smulian theorem.
Indeed assume that \(K\) is a relatively weakly countably compact subset of a Banach space \(X\).
To prove that \(K\) is relatively weakly compact it suffices showing that \(\overline{K}^{w^*} \subset X\).
Assume the contrary and choose sequences \((x_n) \subset K\), \((x_n^*) \subset X^*\) and \(\delta > 0\)
according to Lemma \ref{ES}. Let \(x^* \in X^*\) be a \(w^*\)-cluster point of \((x_n^*)\) and let
\(x \in X\) be a weak-cluster point of \((x_n)\). It follows that \(x^*(x_i) = 0\) for all \(i \in \mathbb{N}\)
and thus \(x^*(x) = 0\). On the other hand, \(|x_i^*(x)| \geq \delta\) for all \(i \in \mathbb{N}\) whence
\(|x^*(x)| \geq \delta\), a contradiction. 
\end{remark}
\begin{Cor} \label{C2}
Let \(K\) be a bounded subset of the Banach space \(X\). We naturally
identify \(X\) with a closed subspace of \(X^{**}\) and set \(L = \overline{K}^{w^*}\). 
Let \(R \colon X^* \to C(L)\) be the natural restriction operator. If \(R\) is weakly compact
then \(K\) is relatively weakly compact.
\end{Cor}
\begin{proof}
Note first that since \(K\) is bounded, \(L\)
is \(w^*\)-compact. We need only show that
\(L \subset X\). If that were not the case then Lemma \ref{ES} yields bounded sequences 
\((x_n) \subset K \),
\((x_n^*) \subset X^*\) and \(\delta > 0\) so that for all \(n \in \mathbb{N}\)
\[ |x_n^*(x_i) | < 1/n , \, \forall \, i < n, \text{ yet }, |x_i^*(x_n) | > \delta,  \, \forall \, i \leq n\]
Let \(f_n = R(x_n^*) \) for all \(n \in \mathbb{N}\).
\(R\) is weakly compact and so \((f_n)\) admits a weak-cluster point \(f \in C(L)\), thanks to the Eberlein-Smulian theorem.
It follows, by the first inequality above, that \(f(x_i) = 0\) for all \(i \in \mathbb{N}\). On the other hand, let
\(x^{**} \in L\) be a \(w^*\)-cluster point of \((x_n)\). Then \(f(x^{**}) = 0\). However, the second inequality above
yields \(|f_i(x^{**})| \geq \delta\) for all \(i \in \mathbb{N}\). Hence, \(|f(x^{**})| \geq \delta\) which is a contradiction.
\end{proof}
The next corollary is a special case of James's compactness theorem (cf. also \cite{M}).
\begin{Cor} \label{Jsep}
Let \(X\) be a Banach space whose dual ball is \(w^*\)-sequentially compact. Let \(K \subset X\) be bounded
with the property that every \(x^* \in X^*\) attains its supremum on \(K\). Then \(K\) is relatively weakly compact.
\end{Cor}
\begin{proof}
\(X\) is naturally identified with a closed subspace of \(X^{**}\). Set \(L = \overline{K}^{w^*}\). Since \(K\) is bounded, \(L\)
is \(w^*\)-compact. Let \(Y = C(L)\), endowed with the supremum norm and consider the natural restriction operator
\(R \colon X^* \to Y\). 
Corollary \ref{C2} will lead the assertion provided we show that \(R\) is weakly compact. To this end,
let \((x_n^*)\) be a bounded sequence in \(X^*\).
Since \(X\) has a \(w^*\)-sequentially compact dual ball we may assume, without loss of generality, that \((x_n^*)\) is \(w^*\)-convergent to some \(x^* \in X^*\).

We set \(f = R(x^*) \) and \(f_n = R(x_n^*)\) for all \(n \in \mathbb{N}\). It follows that
\((f_n - f)\) is a bounded sequence in \(Y\) satisfying \(\lim_n f_n(t) = f(t)\) for all \(t \in K\). Note also that for all \((a_n) \in \ell_1^+\), \(\|\sum_n a_n (f_n - f)\| = 
\sup_{x^{**} \in L} |\sum_n a_n x^{**}(x_n^* - x^*)| = 
\sup_{x \in K} |\sum_n a_n (x_n^* - x^*)(x)| \)
and the latter is attained at an element of \(K\) thanks to our hypothesis. We infer from Corollary \ref{C1}
that \((f_n)\) is weakly convergent to \(f\) and therefore \(R\) is indeed weakly compact by the Eberlein-Smulian theorem.
\end{proof}
\section{Proof of the main result} \label{S2}
\begin{Lem} \label{L2}
Let \(Y\) be a Banach space and \((y_n)\) a normalized sequence in \(Y\) generating \(\ell_1^+\) almost isometrically.
Let \(0 < \delta_0 < 1\). There exist positive integers \(m_0\), \(n_0\) and a sequence of positive scalars \((b_n)\)
so that letting \(z_n = y_{m_0 + n}\), \(n \in \mathbb{N}\), the following hold:
\begin{enumerate}
\item \(\| \sum_{n=1}^{n_0} b_n z_n \| = \delta_0\)
\item \((1+b_n) \|\sum_{i=1}^n b_i z_i \| = \|\sum_{i=1}^{n+1} b_i z_i \| < 1, \,  \forall \, n \geq n_0\)
\item \([\prod_{i=n}^\infty (1 + b_i) ] \|\sum_{i=1}^n b_i z_i \| = \|\sum_{i=1}^\infty b_i z_i \| \leq 1, \, \forall \, n \geq n_0\)
\end{enumerate}
\end{Lem}
\begin{proof}
We first choose \(0 < \epsilon < 1 - \delta_0\) such that \(\delta_0 e^{(1 + \epsilon)(1 + 2 \epsilon - \delta_0)} < 1\).
This is possible because \(\delta_0 e^{1-\delta_0} < 1\). Since \((y_n)\) generates \(\ell_1^+\) almost isometrically, there exists
\(m_0 \in \mathbb{N}\) such that
\[\biggl \| \sum_{i \geq m_0} a_i y_i \biggr \| \geq (1+ \epsilon)^{-1} \sum_{i \geq m_0} a_i , \, \forall \, (a_i)_{i \geq m_0} \in \ell_1^+\]
We next choose \(n_0 \in \mathbb{N}\) such that \( 1 + \epsilon < n_0 \epsilon\) and set \(v_0 = \frac{1}{n_0} \sum_{i=1}^{n_0} z_i\).
Note that \(\|v_0\| \geq (1 + \epsilon)^{-1}\). Define \(b_i = \frac{\delta_0}{n_0 \|v_0\|}\), \(i \leq n_0\), and
\(u_{n_0} = \sum_{i=1}^{n_0} b_i z_i\). Then \(\|u_{n_0}\| = \delta_0\) and \(0 < b_i < \epsilon\), \(i \leq n_0\). Observe that
\[ (1+ b_{n_0}) \|u_{n_0}\| \leq ( 1 + \epsilon) \delta_0 < e^{\epsilon} \delta_0 < \delta_0 e^{1 - \delta_0} < 1\]
We now apply the intermediate value theorem in a manner similar to that in the proof of Lemma \ref{L1} to obtain \(b_{n_0 + 1} > 0\)
such that 
\[(1+ b_{n_0}) \|u_{n_0}\| = \| u_{n_0} + b_{n_0 + 1} z_{n_0 +1} \| < 1\]
We next assume that \(n \geq n_0 +1\) and that we have constructed positive scalars \((b_i)_{i=n_0 + 1}^n \) satisfying
\[(1+b_k) \biggl \|\sum_{i=1}^k b_i z_i \biggr \| = \biggl \|\sum_{i=1}^{k+1} b_i z_i  \biggr \| < 1, \, n_0 \leq k \leq n-1\]
Then,
\begin{align}
(1+b_n) \biggl \|\sum_{i=1}^n b_i z_i \biggr \| &= \biggl [ \prod_{i=n_0}^n ( 1 + b_i) \biggr ] \biggl \|\sum_{i=1}^{n_0} b_i z_i \biggr \|
\leq \delta_0 e^{\sum_{i=n_0}^n b_i} \notag \\
&\leq \delta_0 e^{(1 + \epsilon) \|\sum_{i=n_0}^n b_i z_i \|} \notag
\end{align}
We now observe that
\[\biggl \|\sum_{i=1}^{n_0 -1 } b_i z_i \biggr \| + \biggl \|\sum_{i=n_0}^n b_i z_i \biggr \| \leq \sum_{i=1}^n b_i 
\leq (1 + \epsilon) \biggl \|\sum_{i=1}^n b_i z_i \biggr \| \leq 1 + \epsilon\]
and that 
\[ \biggl \|\sum_{i=1}^{n_0 - 1}  b_i z_i  \biggr \| \geq \biggl \|\sum_{i=1}^{n_0} b_i z_i \biggr \| - b_{n_0} \geq \delta_0 - \epsilon\]
Hence, 
\[\|\sum_{i=n_0}^n b_i z_i \| \leq 1 + \epsilon - (\delta_0 - \epsilon) = 1 + 2 \epsilon - \delta_0\]
We conclude that
\[(1+b_n) \biggl \|\sum_{i=1}^n b_i z_i \biggr \| \leq \delta_0 e^{(1+\epsilon)(1 + 2 \epsilon - \delta_0)} < 1\]
Once again, the intermediate value theorem yields some \(b_{n+1} > 0\) so that
\[(1+b_n) \biggl \|\sum_{i=1}^{n} b_i z_i \biggr \| = \biggl \|\sum_{i=1}^{n+1} b_i z_i  \biggr \| < 1\]
We have thus inductively constructed a sequence of positive scalars \((b_n)\) satisfying \((1)\) and \((2)\).
It follows that \(\sum_n b_n\) is a convergent series and hence \((3)\) is an immediate consequence of \((2)\).
\end{proof}
\begin{Cor} \label{C3}
Let \((y_n)\) be a normalized sequence in the Banach space \(Y\) generating \(\ell_1^+\)
almost isometrically.
Let \(F \subset B_{Y^*}\) be such that
\(\sum_n a_n y_n\) attains its norm at some element of \(F\) for all \((a_n) \in \ell_1^+\). Then for all 
\( 0 < \delta < 1\) there exists \(y^* \in F\) such that
\(\limsup_n |y^* (y_n)| \geq \delta\). Moreover, if every subsequence of \((y_n)\) admits a 
\(\tau_F\)-cluster point which attains its norm at some element of \(F\), then there exists  
\(y^* \in F\) such that
\(\limsup_n |y^* (y_n)| = 1\).
\end{Cor}
\begin{proof}
Assume, to the contrary, that for some \(0 < \delta < 1\) we had that
\(\limsup_n |y^*(y_n)|\) \(< \delta\) for all \(y^* \in F\).
Choose \(0 < \delta_0 < 1\) and \(\epsilon_0 > 0\) so that \(\delta_0 > \delta ( 1 + \epsilon_0) \).
Apply Lemma \ref{L2} to find \(n_0 \in \mathbb{N}\), a sequence of positive scalars \((b_n)\)
and a subsequence \((z_n)\) of \((y_n)\) so that 
\(\|\sum_{i=1}^{n_0} b_i z_i \| = \delta_0\) and
\[(1+b_n) \biggl \|\sum_{i=1}^{n} b_i z_i \biggr \| = \biggl \|\sum_{i=1}^{n+1} b_i z_i  \biggr \| < 1,
\, \forall \, n \geq n_0\]
We set \(u_n = \sum_{i=1}^n b_i z_i \), for all \(n \in \mathbb{N}\),
and \(u = \sum_n b_n z_n\). Then
\[(1 + b_n) \|u_n \| = \|u_{n+1}\| \text{ and } 
\biggl [\prod_{i=n}^\infty (1 + b_i) \biggr ] \|u_n\| = \|u \|, \forall \, n \geq n_0\] 
In particular, \(\|u_n\| < \|u_{n+1}\| < \|u\|\), for all \(n \geq n_0\). 
So if we let
\(t_n = \|u\|/\|u_n\|\) we obtain that \((t_n)_{n \geq n_0}\) is strictly decreasing to \(1\).

Let \(y^* \in F\) be arbitrary and choose \(m > n_0\) so that \(|y^*(z_n)| < \delta\)
and \(b_n < \epsilon_0\) for all \(n > m\). 
We now have the following estimates
\begin{align}
&|y^*(u)| \leq \|u_m \| + \sum_{n=m+1}^\infty b_n |y^*(z_n)| \leq \|u_m\| + \delta \sum_{n=m+1}^\infty b_n \notag \\
&\leq \|u_m\| + \delta (1 + \epsilon_0) \sum_{n=m+1}^\infty \frac{b_n}{1 + \epsilon_0} 
\leq \|u_m\| + \delta (1 + \epsilon_0) \sum_{n=m+1}^\infty \frac{b_n}{1 + b_n} \notag \\
&\leq \|u_m\| +
\delta ( 1 + \epsilon_0) \sum_{n=m+1}^\infty \ln{(1 + b_n)} \notag \\
&= \|u_m \| + \delta (1 + \epsilon_0) \ln{ \biggl [ \prod_{n=m+1}^\infty (1 + b_n) \biggr ] } 
= \|u_m\| + \delta (1 + \epsilon_0) \ln{ \biggl (\frac{\|u\|}{\|u_{m+1}\|} \biggr )} \notag \\
&=\frac{\|u\|}{t_m} + \delta (1 + \epsilon_0) \ln{ (t_{m+1})} \leq \frac{\|u\|}{t_m} + 
\delta (1 + \epsilon_0 )(t_m - 1) \notag \\
&=\|u\| \biggl [ \frac{1}{t_m} + \frac{ \delta (1 + \epsilon_0)}{\|u\|} (t_m -1) \biggr ] 
< \|u\| \biggl [ \frac{1}{t_m} + \frac{ \delta_0 }{\|u\|} (t_m -1) \biggr ] \notag \\
&= \|u\| \biggl [ \frac{1}{t_m} + \frac{ \|u_{n_0}\|}{\|u\|} (t_m -1) \biggr ] 
< \|u\| \biggl [ \frac{1}{t_m} + \frac{ \|u_{m}\|}{\|u\|} (t_m -1) \biggr ] \notag \\ 
&= \|u\| \biggl ( \frac{1}{t_m} + \frac{t_m -1}{t_m} \biggr ) = \|u\| \notag
\end{align}
Therefore, \(|y^*(u)| < \|u\|\) for all \(y^* \in F\) contradicting our hypothesis that
\(u\) attains its norm at some element of \(F\).

For the moreover assertion, let us suppose that \(\limsup_n |y^*(y_n)| < 1\)
for all \(y^* \in F\). It follows that \(|y^*(y) | < 1\) for all \(y^* \in F\)
and every \(\tau_F\)-cluster point \(y\) of \((y_n)\). We deduce from this and
our hypothesis that \(\|y\| < 1\) for every \(\tau_F\)-cluster point \(y\) of \((y_n)\)
that attains its norm at some element of \(F\). Successive applications of the first part of this corollary
now yield a nested sequence \(M_1 \supset M_2 \supset \cdots \) of infinite subsets of \(\mathbb{N}\)
and a sequence \((y_n^*) \subset F\) so that
\begin{equation} \label{E1}
|y_n^*(y_i)| > 1 - \frac{1}{n}, \, \forall \, i \in M_n, \, \forall \, n \in \mathbb{N}
\end{equation}
We next choose integers \(m_1 < m_2 < \cdots \) with \(m_n \in M_n\) for all \(n \in \mathbb{N}\).
Our hypothesis now yields a \(\tau_F\)-cluster point \(y_0\) of \((y_{m_n})\) which attains its
norm at some element of \(F\). Therefore, \(\|y_0\| < 1\) by our comments in the beginning of this paragraph. 
However, \eqref{E1} implies that \(\|y_0\| \geq 1\). This contradiction completes the proof of the lemma.
\end{proof}
\begin{Def}
Let \(Y\) be a Banach space and \(F \subset B_{Y^*}\). A subset \(K\) of \(Y\) is said to be
\(F\)-admissible if \(K\) is bounded, \(\tau_F\)-compact and whenever \((z_n) \subset K\) and
\((a_n) \in \ell_1\) then \(\sum_n a_n z_n\) attains its norm at some element of \(F\).
\end{Def}
\begin{Lem} \label{L3}
Let \(Y\) be a Banach space, \(F \subset B_{Y^*}\) and let \(K \subset Y\) be
\(F\)-admissible. Let \((y_n) \subset K\) and let \((I_n)\) be a sequence of finite subsets of \(\mathbb{N}\)
of the same cardinality. Let \((\lambda_n)\) be a bounded sequence of scalars and set \(u_n = \sum_{i \in I_n} \lambda_i y_i\)
for all \(n \in \mathbb{N}\). Assume that \((u_n)\) is normalized and generates \(\ell_1^+\) almost isometrically. Then there exists
\(y^* \in F\) such that \(\limsup_n |y^*(u_n)| =1\).
\end{Lem}
\begin{proof}
It is not hard to see that the \(F\)-admissibility of \(K\) implies that 
\(\sum_n a_n u_n\) attains its norm at an element of \(F\) for every \((a_n) \in \ell_1\).
Therefore, in view of Corollary \ref{C3}, it suffices showing that every subsequence of \((u_n)\) admits a \(\tau_F\)-cluster point which attains its norm
at some element of \(F\). Note also that our assumptions on \((u_n)\) are also satisfied by any of its subsequences. Thus, we need only establish our assertion
for \((u_n)\) solely. To this end, let \(d \in \mathbb{N}\) be such that \(|I_n| = d\) for all \(n \in \mathbb{N}\). Let us denote by \(m(n,i)\) the \(i\)-th element
of \(I_n\) for all \(n \in \mathbb{N}\) and \(i \leq d\). Since \((\lambda_n)\) is bounded, by passing to a subsequence of \((u_n)\) if necessary,
there is no loss of generality in assuming that
\(\lim_n \lambda_{m(n,i)} = \mu_i \in \mathbb{C}\) for all \(i \leq d\).

Since \(K\) is \(\tau_F\)-compact, \(K^d\) is compact in the product topology induced by \(\tau_F\).
Define \(\vec{y_n} = \bigl ( y_{m(n,i)})_{i=1}^d\), for all \(n \in \mathbb{N}\), and choose a cluster point
\((z_i)_{i=1}^d \in K^d\) of \((\vec{y_n})\). Clearly, \(\sum_{i=1}^d \mu_i z_i\) is a \(\tau_F\)-cluster point of
\((u_n)\) which attains its norm at some element of \(F\). 
\end{proof}
In the sequel we shall make use of the following simple observation: Suppose that \((e_n)\) is a sequence in a Banach space
isometrically equivalent to the \(\ell_1\)-basis. Let \(d \in \mathbb{N}\), \(m_1 < \dots < m_d\) in \(\mathbb{N}\)
and non-zero scalars \((\lambda_i)_{i=1}^{d+1}\) with \(\sum_{i=1}^{d+1} |\lambda_i| = 1\).
Define \(u_n = \sum_{i=1}^d \lambda_i e_{m_i} + \lambda_{d+1} e_n\) for all \(n > m_d\).
Then \((u_n)_{n > m_d}\) generates \(\ell_1^+\) isometrically.
\begin{Lem} \label{L4}
Let \(Y\) be a Banach space, \(F \) a symmetric subset of \(B_{Y^*}\) and let \(K \subset Y\) be
\(F\)-admissible. Then \(K\) contains no sequence isometrically equivalent to the \(\ell_1\)-basis.
\end{Lem}
\begin{proof}
Assume to the contrary that \((y_n)\) is a normalized sequence isometrically equivalent to 
the \(\ell_1\)-basis. Notice that \((\frac{1}{2} y_1 - \frac{1}{2} y_n)_{n \geq 2}\) generates
\(\ell_1^+\) isometrically. Since \(F\) is symmetric, Lemma \ref{L3} yields \(M_1 \in [\mathbb{N}]\),
\(\min M_1 > 1\), and \(y_1^* \in F\) so that 
\(\lim_{n \in M_1} y_1^*( \frac{1}{2} y_1 - \frac{1}{2} y_n) =1\). Set \(m_1 = 1\) and
\(m_2 = \min M_1\). Then
\((\frac{1}{4} y_{m_1} + \frac{1}{4} y_{m_2} - \frac{1}{2} y_n)_{n \in M_2 \setminus \{m_2\}}\)
generates \(\ell_1^+\) isometrically. So we choose, by Lemma \ref{L3} and the symmetry of \(F\), \(y_2^* \in F\)
and \(M_2 \in [M_1 \setminus \{m_2\}]\) so that
\(\lim_{n \in M_2} y_2^*( \frac{1}{4} y_{m_1} + \frac{1}{4} y_{m_2} - \frac{1}{2} y_n)=1\).

Continuing in this manner we inductively construct a nested sequence 
\(M_1 \supset M_2 \supset \cdots\) of infinite subsets of \(\mathbb{N}\) so that letting
\(m_n = \min M_{n-1}\), \(n \geq 2\), then \(m_n \notin M_n\) and there exists \(y_n^* \in F\)
so that
\[\lim_{k \in M_n} y_n^* \biggl ( \sum_{i=1}^n \frac{1}{2n} y_{m_i} - \frac{1}{2} y_k \biggr ) = 1,
\, \forall \, n \in \mathbb{N}\]
Let \(y_0 \in K\) be a \(\tau_F\)-cluster point of \((y_{m_n})\). Observe that
\(y_n^*( \frac{1}{2} y_{m_i} - \frac{1}{2} y_0 ) = 1\), for all \(i \leq n\) and \(n \in \mathbb{N}\).
It follows from this that \(( \frac{1}{2} y_{m_n} - \frac{1}{2} y_0 )_{n=1}^\infty\) generates
\(\ell_1^+\) isometrically. Finally, let \((a_n)\) be any sequence of positive scalars such that
\(\sum_n a_n =1\). We infer from our assumptions that 
\(\sum_n a_n ( \frac{1}{2} y_{m_n} - \frac{1}{2} y_0 )\) attains its norm at some element of \(F\).
Since \(F\) is symmetric, there exists \(z^* \in F\) so that
\(\sum_n a_n z^*( \frac{1}{2} y_{m_n} - \frac{1}{2} y_0 ) =1\), whence
\(z^*(y_{m_n}) - z^*(y_0) = 2\) for all \(n \in \mathbb{N}\). But this contradicts the fact that 
\(y_0\) is a \(\tau_F\)-cluster point of \((y_{m_n})\).
\end{proof}
\begin{Lem} \label{L5}
Let \(Y\) be a Banach space, \(F \) a symmetric subset of \(B_{Y^*}\) and let \(K \subset Y\) be
\(F\)-admissible. Assume that
\((y_n)\) is a normalized sequence in \(K\) generating \(\ell_1\)
almost isometrically. Then given \(r \in \mathbb{N}\), \(\Delta \subset \mathbb{T}\) finite and
a collection \((N_i)_{i=1}^r \) of pairwise disjoint members of \([\mathbb{N}]\), there exists 
a collection \((P_i)_{i=1}^r\)
with \(P_i \in [N_i]\) for all \(i \leq r\), satisfying the following property:
for every choice \(\theta_1, \dots, \theta_r\) of members of \(\Delta\) there exists \(y^* \in F\)
so that \(\lim_{k \in P_i} y^*(y_k) = \theta_i\) for all \(i \leq r\).
\end{Lem}
\begin{proof}
Fix \(\theta_1, \dots, \theta_r\) in \(\Delta\). It is sufficient to find \(M_i \in [N_i]\),
\(i \leq r\), and \(y_0^* \in F\) so that
\(\lim_{k \in M_i} y_0^*(y_k) = \theta_i\) for all \(i \leq r\). Since \(\Delta^r\) is finite, by repeating
this process a finite number of times we shall arrive at the desired choices of \(P_1, \dots, P_r\).
We first choose a sequence \(I_1 < I_2 < \cdots < \) of successive finite subsets of \(\mathbb{N}\)
of the form \(I_j = \{l_{j1}, <  \cdots , < l_{jr} \}\), where \(l_{ji} \in N_i\) for all \(i \leq r\)
and \(j \in \mathbb{N}\). We define \(v_k = \sum_{i=1}^r \overline{\theta_i} y_{l_{ki}}\) 
and \(u_k = \frac{v_k}{\|v_k\|}\) for all \(k \in \mathbb{N}\). Since \((y_k)\) generates \(\ell_1\)
almost isometrically, by passing to a subsequence of \((v_k)\)
if necessary, there is no loss of generality in assuming that \(\lim_k \|v_k\| = r\). For the same reason we may
also assume that \((u_k)\) generates \(\ell_1\) almost isometrically. It is clear that \((u_k)\)
fulfills the hypotheses in Lemma \ref{L3} and hence there exists \(y_1^* \in F\) so that
\(\limsup_k |y_1^*(u_k)| =1\). Once again, by passing to a subsequence of \((u_k)\) if necessary,
we may assume that \(\lim_k |y_1^*(u_k)| =1\) and that \(\lim_k y_1^*(y_{l_{ki}}) =a_i \in \mathbb{C}\),
for all \(i \leq r\). 

It follows now that \(|\sum_{i=1}^r a_i \overline{\theta_i} | = r\). We deduce from this and the fact that
\(|a_i \theta_i| \leq 1\) for all \(i \leq r\) that there is some \(z \in \mathbb{T}\) such that
\(a_i \overline{\theta_i} = z\) for all \(i \leq r\). Set \(M_i = \{l_{ki} : \, k \in \mathbb{N} \}\)
for all \(i \leq r\) and \(y_0^* = \overline{z} y_1^*\). Since \(F\) is symmetric \(y_0^* \in F\)
and it is clear that \(\lim_{k \in M_i} y_0^* ( y_k ) = \theta_i\) for all \(i \leq r\).
\end{proof}
\begin{Lem} \label{L6}
Let \(Y\) be a Banach space, \(F \) a symmetric subset of \(B_{Y^*}\) and let \(K \subset Y\) be
\(F\)-admissible. Assume that
\((y_n)\) is a normalized sequence in \(K\) generating \(\ell_1\)
almost isometrically. Then there exists a subset of \(K\) of cardinality equal to the continuum whose
elements are \(\tau_F\)-cluster points of \((y_n)\), generating \(\ell_1(\mathbf{c})\) isometrically.
\end{Lem}
\begin{proof}
Let \((\epsilon_n)\) be a scalar sequence strictly decreasing to \(0\) and choose an increasing sequence
\(\Delta_1 \subset \Delta_2 \subset \cdots \) of finite subsets of \(\mathbb{T}\) such that \(\Delta_n\)
is an \(\epsilon_n\)-net for \(\mathbb{T}\) for all \(n \in \mathbb{N}\).
Successive applications of Lemma \ref{L5} yield a tree \((M_\alpha)_{\alpha \in \mathcal{T}}\) of
infinite subsets of \(\mathbb{N}\) with the following property: For each \(n \in \mathbb{N}\) and all choices
\((z_\alpha)_{\alpha \in \{0,1\}^n}\) of elements from \(\Delta_n\) there exists \(y^* \in F\) so that
\(\lim_{k \in M_\alpha} y^*(y_k) = z_\alpha\) for all \(\alpha \in \{0,1\}^n\).

Let \(\mathfrak{b} = \{\beta_1, < \beta_2, < \dots \}\) be a branch of \(\mathcal{T}\).
Let \(N_{\mathfrak{b}} \in [\mathbb{N}]\) be almost contained in each \(M_{\beta_j}\) (i.e, 
\(N_{\mathfrak{b}} \setminus M_{\beta_j}\) is finite for all \(j \in \mathbb{N}\)).
Let \(y_{\mathfrak{b}} \in K\) be a \(\tau_F\)-cluster point of \((y_n)_{n \in N_{\mathfrak{b}}}\).
If \(\mathfrak{B}\) denotes the set of all branches of \(\mathcal{T}\) then we
claim that \((y_{\mathfrak{b}})_{\mathfrak{b} \in \mathfrak{B}}\) 
generates \(\ell_1(\mathbf{c})\) isometrically.

Indeed, suppose that \(m \in \mathbb{N}\) and that \(\mathfrak{b}^1, \dots , \mathfrak{b}^m\)
are distinct branches of \(\mathcal{T}\). Write
\(\mathfrak{b}^i = \{\beta_1^i, < \beta_2^i, < \dots \}\), \(i \leq m\). Given \(p \in \mathbb{N}\)
choose an integer \(n_0 > p\) so that \(\beta_{n_0}^i \ne \beta_{n_0}^j\) whenever \(i \ne j\) 
in \(\{1, \dots , m\}\). Let \((a_i)_{i=1}^m\) be scalars and write \(a_i = |a_i| z_i \) with \(z_i \in \mathbb{T}\)
for all \(i \leq m\). We then choose, for all \(i \leq m\), \(\theta_i \in \Delta_{n_0}\) so that
\(|\theta_i - \overline{z}_i | < \epsilon_{n_0}\). Our construction yields \(y^* \in F\) such that
\(\lim_{k \in M_{\beta_{n_0}^i}} y^*(y_k) = \theta_i\) for all \(i \leq m\).
It follows that \(\lim_{k \in N_{\mathfrak{b}^i}} y^*(y_k) = \theta_i\) for all \(i \leq m\) and therefore
\(y^*(y_{\mathfrak{b}^i}) = \theta_i\) for all \(i \leq m\). 
We infer now from the \(F\)-admissibility of \(K\) that \( \| y_{\mathfrak{b}^i} \| = 1\) for all \(i \leq m\).
Finally we have the estimates
\begin{align}
&\biggl \| \sum_{i=1}^m a_i y_{\mathfrak{b}^i} \biggr \| \geq 
\biggl | \sum_{i=1}^m a_i y^*(y_{\mathfrak{b}^i}) \biggr | =
\biggl | \sum_{i=1}^m a_i \theta_i \biggr | \notag \\
&= \biggl | \sum_{i=1}^m a_i \overline{z}_i - \sum_{i=1}^m a_i (\overline{z}_i - \theta_i ) \biggr | \geq
\sum_{i=1}^m |a_i|  - \sum_{i=1}^m |a_i| \epsilon_{n_0} \notag \\
&\geq (1 - \epsilon_p) \sum_{i=1}^m |a_i| \notag
\end{align}
Since \(p \in \mathbb{N}\) is arbitrary and \(\lim_n \epsilon_n = 0\), we conclude that
\(\| \sum_{i=1}^m a_i y_{\mathfrak{b}^i} \| = \sum_{i=1}^m |a_i|\) for every \(m \in \mathbb{N}\)
and all choices of scalars \((a_i)_{i=1}^m\). Hence, 
\((y_{\mathfrak{b}})_{\mathfrak{b} \in \mathfrak{B}}\) 
generates \(\ell_1(\mathbf{c})\) isometrically.
\end{proof}
\begin{Cor} \label{C4}
Let \(Y\) be a Banach space, \(F \) a symmetric subset of \(B_{Y^*}\) and let \(K \subset Y\) be
\(F\)-admissible. Suppose that \(Z\) is a closed linear subspace of \(Y\) isomorphic to \(\ell_1\).
Then \(B_Z \setminus K \ne \emptyset\).
\end{Cor}
\begin{proof}
Let us assume that \(B_Z \subset K\). James's distortion theorem \cite{J2} now yields a normalized sequence \((y_n)\)
in \(K\) generating \(\ell_1\) almost isometrically. We deduce from Lemma \ref{L6} that \(K\) contains a
normalized sequence isometrically equivalent to the \(\ell_1\)-basis, contradicting Lemma \ref{L4}.
\end{proof}
\begin{proof}[Proof of Theorem \ref{MTh}]
There is no loss of generality in assuming that \(F\) is symmetric. Suppose that \((x_n^*)\) is a bounded sequence in \(X^*\)
such that \((Tx_n^*)\) is equivalent to the \(\ell_1\)-basis. It follows that \((x_n^*)\) is also equivalent to the
\(\ell_1\)-basis. Let \(Z\) denote the closed linear subspace of \(Y\) generated by \((Tx_n^*)\). It is clear now that
there exists a closed ball \(B^* \subset X^*\), centered at the origin, so that \(B_Z \subset TB^*\). But \(T\) is
\((w^*, \tau_F)\) continuous and so \(K = TB^*\) is \(\tau_F\)-compact. \(K\) is also bounded because \(T\) is. Since every
vector in the range of \(T\) attains its norm at an element of \(F\), we obtain that \(K\) is \(F\)-admissible. However,
\(Z\) is isomorphic to \(\ell_1\) and \(B_Z \subset K\) contradicting Corollary \ref{C4}.

We infer from the above that for every bounded sequence \((x_n^*)\) in \(X^*\), \((Tx_n^*)\) admits no \(\ell_1\)-subsequence.
Rosenthal's \(\ell_1\)-Theorem \cite{R} now yields a weak Cauchy subsequence \((Tx_{k_n}^*)\) of \((Tx_n^*)\). Let \(x_0^* \in X^*\)
be a \(w^*\)-cluster point of \((x_{k_n}^*)\). The \((w^*, \tau_F)\) continuity of \(T\) implies that \(Tx_0^*\) is a \(\tau_F\)-cluster point of \((Tx_{k_n}^*)\).
Since the latter sequence is weak Cauchy, it follows that \(\lim_n y^* ( Tx_{k_n}^*) = y^* (Tx_0^*)\) for all \(y^* \in F\).
On the other hand, \(\sum_n a_n ( Tx_{k_n}^* - Tx_0^*) \in \mathrm{ Im }(T)\) for all \((a_n) \in \ell_1^+\) and therefore        
we deduce from our assumptions on \(T\), that it attains its norm at an element of \(F\). Corollary \ref{C1} now yields that 
\(( Tx_{k_n}^*)\) is weakly convergent to \(Tx_0^*\). We conclude from the Eberlein-Smulian theorem that \(T\) maps
\(w^*\)-compact subsets of \(X^*\) to weakly compact subsets of \(Y\). In particular, \(T\) is weakly compact.

We next show that \(T\) is \((w^*, w)\) continuous. To this end, we first claim that if \(V \subset Y\) is norm-closed and convex 
then \(T^{-1}V\) is a \(w^*\)-closed subset of \(X^*\). Indeed, \(T^{-1}V\) is convex and thus, by the Krein-Smulian theorem for the \(w^*\) topology \cite{KS}
(cf. also \cite{DS}),
it suffices showing that \(U^* \cap T^{-1}V\) is \(w^*\)-closed for every closed ball \(U^*\) in \(X^*\) centered at the origin. So let \((u_\lambda^*)_{\lambda \in \Lambda}\)
be a net in \(U^* \cap T^{-1}V\) \(w^*\)-converging to some \(u^* \in X^*\). 
Clearly, \(u^* \in U^*\) and so we need to show that \(Tu^* \in V\). We deduce from the first part of the proof, that
\(TU^*\) is a weakly compact subset of \(Y\). Therefore, there is no loss of generality, by passing to a subnet if necessary, to assume that  
\((Tu_\lambda^*)_{\lambda \in \Lambda}\) weakly converges to some \(y_0 \in TU^*\). Since \(V\) is weakly closed, by Mazur's theorem, we obtain that 
\(y_0 \in V \cap TU^*\). Note that \((Tu_\lambda^*)_{\lambda \in \Lambda}\) \(\tau_F\)-converges to \(Tu^*\) as \(T\) is \((w^*, \tau_F)\) continuous.
It follows that \(y^*(Tu^*) = y^*(y_0)\) for all \(y^* \in F\). Note also that \(F\) separates points in \(\mathrm{ Im }(T)\). Since
both \(Tu^*\) and \(y_0\) belong to \(TU^*\), we are led to the identity
\(Tu^* = y_0 \in V\) which proves our claim. 

We next claim that if \((x_\lambda^*)_{\lambda \in \Lambda}\) is a net in \(X^*\) \(w^*\)-converging to the origin, then
\((Tx_\lambda^*)_{\lambda \in \Lambda}\) is weakly converging to the origin in \(Y\). Were this false, we would find
\(z^* \in Y^*\), \(\delta > 0\) and a subnet \((Tx_\mu^*)_{\mu \in M}\) of \((Tx_\lambda^*)_{\lambda \in \Lambda}\) 
so that \(|z^* (Tx_\mu^*) | \geq \delta\) for all \(\mu \in M\). Without loss of generality, by replacing \(z^*\) by \(\theta z^*\) for a suitable
\(\theta \in \mathbb{T}\) and passing to a further subnet if necessary,
we may assume that \(Tx_\mu^* \in V_0\) for all \(\mu \in M\), where we have set
\(V_0 = \{ y \in Y : \, \mathrm{ Re } [z^* (y)] \geq \delta /2\}\). It is clear that
\(V_0\) is a norm-closed and convex subset of \(Y\) and hence our initial claim yields that
\(T^{-1}V_0\) is a \(w^*\)-closed subset of \(X^*\) containing the subnet \((x_\mu^*)_{\mu \in M}\).
However, this subnet is \(w^*\)-convergent to the origin of \(X^*\) and thus \(V_0\) contains the origin of \(Y\) which is absurd.
Thus, our claim holds and so \(T\) is indeed \((w^*, w)\) continuous.
\end{proof}
\begin{Cor}[James's compactness Theorem] \label{Jamsup}
Let \(X\) be a Banach space and let \(K \subset X\) be bounded. Assume that every \(x^* \in X^*\) attains its supremum on \(K\).
Then \(K\) is relatively weakly compact.
\end{Cor}
\begin{proof}
We naturally
identify \(X\) with a closed subspace of \(X^{**}\) and set \(L = \overline{K}^{w^*}\). 
Let \(R \colon X^* \to C(L)\) be the natural restriction operator. Let \(F = \{\delta_x : \, x \in K\}\), where \(\delta_x\) stands for the Dirac measure at \(x \in K\).
Clearly, \(R\) is \((w^*, \tau_F)\) continuous. Our assumptions on \(K\) imply that every vector in the range of \(R\) attains its norm at an element of \(F\).
We deduce from Theorem \ref{MTh} that
\(R\) is weakly compact and hence, \(K\) is relatively weakly compact by Corollary \ref{C2}.
\end{proof}

\end{document}